\documentclass[12pt]{amsart}
\reversemarginpar
\usepackage{amsmath}
\usepackage{amssymb}
\usepackage{graphicx}
\usepackage{graphicx}
\usepackage{dcolumn}
\usepackage{bm}
\usepackage{amsfonts}
\usepackage{latexsym}
\usepackage{color}

\begin{document}
\newcommand{\commentout}[1]{}

\newcommand{\nwc}{\newcommand}
\nwc{\red}{\color{red}}
\nwc{\blue}{\color{blue}}
\newcommand{\bz}{{\mathbf z}}
\newcommand{\sqk}{\sqrt{\ks}}
\newcommand{\sqkone}{\sqrt{|\k\alpha|}}
\newcommand{\sqktwo}{\sqrt{|\k\beta|}}
\newcommand{\invsqkone}{|\k\alpha|^{-1/2}}
\newcommand{\invsqktwo}{|\k\beta|^{-1/2}}
\newcommand{\partz}{\frac{\partial}{\partial z}}
\newcommand{\grady}{\nabla_{\by}}
\newcommand{\gradp}{\nabla_{\bp}}
\newcommand{\gradx}{\nabla_{\bx}}
\newcommand{\invf}{\cF^{-1}_2}
\newcommand{\myphi}{\tilde\Theta_{(\eta,\rho)}}
\newcommand{\minrg}{|\min{(\rho,\phi^{-1})}|}
\newcommand{\al}{\alpha}
\newcommand{\xvec}{\vec{\mathbf x}}
\newcommand{\kvec}{{\vec{\mathbf k}}}
\newcommand{\lt}{\left}
\newcommand{\ksq}{\sqrt{\ks}}
\newcommand{\rt}{\right}
\newcommand{\ga}{\phi}
\newcommand{\vas}{\varepsilon}
\newcommand{\lan}{\left\langle}
\newcommand{\ran}{\right\rangle}
\newcommand{\tvas}{{W_z^\vas}}
\newcommand{\psiep}{{W_z^\vas}}
\newcommand{\wep}{{W^\vas}}
\newcommand{\weptil}{{\tilde{W}^\vas}}
\newcommand{\wepz}{{W_z^\vas}}
\newcommand{\weps}{{W_s^\ep}}
\newcommand{\wepsp}{{W_s^{\ep'}}}
\newcommand{\wepzp}{{W_z^{\vas'}}}
\newcommand{\wepztil}{{\tilde{W}_z^\vas}}
\newcommand{\vvas}{{\tilde{\ml L}_z^\vas}}
\newcommand{\veptil}{{\tilde{\ml L}_z^\vas}}
\newcommand{\vep}{{{ V}_z^\vas}}
\newcommand{\cvc}{{{\ml L}^{\ep*}_z}}
\newcommand{\cvcp}{{{\ml L}^{\ep*'}_z}}
\newcommand{\cvp}{{{\ml L}^{\ep*'}_z}}
\newcommand{\cvtil}{{\tilde{\ml L}^{\ep*}_z}}
\newcommand{\cvtilp}{{\tilde{\ml L}^{\ep*'}_z}}
\newcommand{\vtil}{{\tilde{V}^\ep_z}}
\newcommand{\ktil}{\tilde{K}}
\newcommand{\n}{\nabla}
\newcommand{\tkappa}{\tilde\kappa}
\newcommand{\Om}{{\Omega}}
\newcommand{\bx}{\mb x}
\nwc{\bv}{\mb v}
\newcommand{\br}{\mb r}
\nwc{\bH}{{\mb H}}
\newcommand{\bu}{\mathbf u}
\nwc{\bxp}{{{\mathbf x}}}
\nwc{\byp}{{{\mathbf y}}}
\newcommand{\bD}{\mathbf D}
\nwc{\bS}{\mathbf S}
\newcommand{\bA}{\mathbf \Phi}
\nwc{\cO}{\mathcal{O}}
\nwc{\co}{\mathcal{o}}
\nwc{\bG}{{\mathbf G}}
\nwc{\bF}{{\mathbf F}}
\nwc{\bd}{{\mathbf d}}
\nwc{\bdhat}{\hat{\mathbf d}}
\nwc{\bR}{{\mathbf R}}
\nwc{\bh}{\mathbf h}
\nwc{\bj}{{\mb j}}
\newcommand{\bB}{\mathbf B}
\newcommand{\bC}{\mathbf C}
\newcommand{\bp}{\mathbf p}
\newcommand{\bq}{\mathbf q}
\newcommand{\by}{\mathbf y}
\nwc{\bI}{\mathbf I}
\nwc{\bP}{\mathbf P}
\nwc{\bs}{\mathbf s}
\nwc{\bX}{\mathbf X}
\nwc{\bZ}{\mathbf Z}
\newcommand{\pdg}{\bp\cdot\nabla}
\newcommand{\pdgx}{\bp\cdot\nabla_\bx}
\newcommand{\one}{1\hspace{-4.4pt}1}
\newcommand{\corr}{r_{\eta,\rho}}
\newcommand{\rinf}{r_{\eta,\infty}}
\newcommand{\rzero}{r_{0,\rho}}
\newcommand{\rzeroinf}{r_{0,\infty}}
\nwc{\om}{\omega}
\nwc{\thetatil}{{\tilde\theta}}

\nwc{\nwt}{\newtheorem}
\nwc{\xp}{{x^{\perp}}}
\nwc{\yp}{{y^{\perp}}}
\nwt{remark}{Remark}
\nwt{lemma}{Lemma}
\nwt{corollary} {Corollary}
\nwt{definition}{Definition} 

\nwc{\ba}{{\mb a}}
\nwc{\bal}{\begin{align}}
\nwc{\be}{\mathbf{e}}
\nwc{\ben}{\begin{equation*}}
\nwc{\bea}{\begin{eqnarray}}
\nwc{\beq}{\begin{eqnarray}}
\nwc{\bean}{\begin{eqnarray*}}
\nwc{\beqn}{\begin{eqnarray*}}
\nwc{\beqast}{\begin{eqnarray*}}

\nwc{\eal}{\end{align}}
\nwc{\ee}{\end{equation}}
\nwc{\een}{\end{equation*}}
\nwc{\eea}{\end{eqnarray}}
\nwc{\eeq}{\end{eqnarray}}
\nwc{\eean}{\end{eqnarray*}}
\nwc{\eeqn}{\end{eqnarray*}}
\nwc{\eeqast}{\end{eqnarray*}}

\nwc{\ep}{\varepsilon}
\nwc{\eps}{\varepsilon}
\nwc{\ept}{\ep }
\nwc{\vrho}{\varrho}
\nwc{\orho}{\bar\varrho}
\nwc{\ou}{\bar u}
\nwc{\vpsi}{\varpsi}
\nwc{\lamb}{\lambda}
\nwc{\Var}{{\rm Var}}

\nwt{proposition}{Proposition}
\nwt{theorem}{Theorem}
\nwt{summary}{Summary}
\nwc{\nn}{\nonumber}
\nwc{\mf}{\mathbf}
\nwc{\mb}{\mathbf}
\nwc{\ml}{\mathcal}

\nwc{\IA}{\mathbb{A}} 
\nwc{\bi}{\mathbf i}
\nwc{\bo}{\mathbf o}
\nwc{\IB}{\mathbb{B}}
\nwc{\IC}{\mathbb{C}} 
\nwc{\ID}{\mathbb{L}} 
\nwc{\IM}{\mathbb{M}} 
\nwc{\IP}{\mathbb{P}} 
\nwc{\II}{\mathbb{I}} 
\nwc{\IE}{\mathbb{E}} 
\nwc{\IF}{\mathbb{F}} 
\nwc{\IG}{\mathbb{G}} 
\nwc{\IN}{\mathbb{N}} 
\nwc{\IQ}{\mathbb{Q}} 
\nwc{\IR}{\mathbb{R}} 
\nwc{\IT}{\mathbb{T}} 
\nwc{\IZ}{\mathbb{Z}} 
\nwc{\pdfi}{{f^{\rm i}}}
\nwc{\pdfs}{{f^{\rm s}}}
\nwc{\pdfii}{{f_1^{\rm i}}}
\nwc{\pdfsi}{{f_1^{\rm s}}}
\nwc{\chis}{{\chi^{\rm s}}}
\nwc{\chii}{{\chi^{\rm i}}}
\nwc{\cE}{{\ml E}}
\nwc{\bE}{{\mathbf E}}
\nwc{\cP}{{\ml P}}
\nwc{\cQ}{{\ml Q}}
\nwc{\cL}{{\ml L}}
\nwc{\cX}{{\ml X}}
\nwc{\cW}{{\ml W}}
\nwc{\cZ}{{\ml Z}}
\nwc{\cR}{{\ml R}}
\nwc{\cV}{{\ml V}}
\nwc{\cT}{{\ml T}}
\nwc{\crV}{{\ml L}_{(\delta,\rho)}}
\nwc{\cC}{{\ml C}}
\nwc{\cA}{{\ml A}}
\nwc{\cK}{{\ml K}}
\nwc{\cB}{{\ml B}}
\nwc{\cD}{{\ml D}}
\nwc{\cF}{{\ml F}}
\nwc{\cS}{{\ml S}}
\nwc{\cM}{{\ml M}}
\nwc{\cG}{{\ml G}}
\nwc{\cH}{{\ml H}}
\nwc{\bk}{{\mb k}}
\nwc{\cbz}{\overline{\cB}_z}
\nwc{\supp}{{\hbox{supp}}}
\nwc{\fR}{\mathfrak{R}}
\nwc{\bY}{\mathbf Y}
\newcommand{\mbr}{\mb r}
\nwc{\pft}{\cF^{-1}_2}
\nwc{\bU}{{\mb U}}
\nwc{\mi}{{\rm i}}
\nwc{\bT}{{\mb T}}
\nwc{\IL}{\mathbb{L}}
\nwc{\mbx}{\mathbf{X}}
\nwc{\bb}{\mathbf{Y}}
\nwc{\bPhi}{\mathbf{\Phi}}
\nwc{\mbe}{\mathbf{E}}
\newcommand{\ind}{\operatorname{I}}

\title{
TV-min and Greedy Pursuit for Constrained Joint Sparsity
and Application to Inverse Scattering}
\author{Albert  Fannjiang}
\email{
fannjiang@math.ucdavis.edu}

       \address{
   Department of Mathematics,
    University of California, Davis, CA 95616-8633}
    
    \begin{abstract}
   This paper proposes a general framework  for compressed sensing of constrained joint sparsity (CJS) which  includes  total variation minimization (TV-min) as an example.  TV- and 2-norm error bounds, independent of the ambient dimension,  are derived 
for the CJS version of Basis Pursuit and Orthogonal Matching Pursuit. 
As an application the
   results  extend  Cand\`es, Romberg  and Tao's 
 proof of exact recovery
of piecewise constant objects
with noiseless incomplete Fourier data to the case of noisy data.

    \end{abstract}
    \maketitle 
    
    \section{Introduction}
    
 One of the most significant developments in imaging and signal processing of the last decade  is compressive sensing (CS) which promises reconstruction with fewer data  
 than the ambient dimension.
The  CS capability \cite{ CT, Don} hinges  on  favorable sensing matrices and enforcing 
    a key  prior knowledge, i.e. sparse objects. 
    
    Consider the linear inverse problem $Y=\bA X+E$ where
$X\in \IC^m$ is  the {\em sparse} object vector  to be recovered,
$Y\in \IC^n$ is the measurement data vector and
$E\in \IC^n$ represents  the (model or external) errors.
The great insight of CS is that the  sparseness of $X$, as measured by the sparsity 
$\|X\|_0\equiv \#$ nonzero elements in $X$, can be effectively enforced by 
L1-minimization (L1-min)  \cite{CDS, DH}
\beq
\min \|Z\|_1 \quad \text{subject to (s.t.)} \quad  \|\bA Z-Y  \|_2\leq \|E\|_2
\label{L1}
\eeq
with favorable sensing matrices $\bPhi$. 

The L1-min idea dates back to  geophysics
research in 1970's \cite{CM, TBM}. 
The L1-minimizer is often a much better approximation to
the sparse object than the traditional minimum energy solution 
via $L2$-minimization because $1$-norm  is closer to
$\|\cdot\|_0$ than the $2$-norm.  Moreover, the L1-min principle is a convex optimization problem
and can be efficiently computed. The L1-min principle 
is effective in recovering the sparse object  with the number of data $n$ much less than
the ambient dimension $m$ if the sensing matrix
$\bA$ satisfies some favorable conditions such as
the restricted isometry property (RIP) \cite{CT}:
$\bA$ is said to satisfy RIP of order $k$ if
\beq 
\label{rip1}
(1-\delta_{k})\|Z\|_{2,2}^2\leq
\|\bA Z\|_{2}^2\leq (1+\delta_{k})\|Z\|_{2}^2
\eeq
for any $k$-sparse vector $Z$ where the minimum of such constant $\delta_k$ is
the restricted isometry constant (RIC) of order $k$. 

The drawback of 
 RIP is that only a few special types of matrices are known to satisfy RIP, including independently and identically  distributed (i.i.d.) random matrices and random partial Fourier matrices
formed by random row selections  of the
discrete Fourier transform. 

A more practical alternative  CS criterion is furnished by the 
 incoherence property as measured by one minus  the mutual coherence
\beq
\label{mu}
\mu({\mathbf \Phi})=\max_{i\neq j} {\lt|\sum_{k}\Phi^*_{ik}
\Phi_{kj}\rt| \over \sqrt{\sum_{k}|\Phi_{ki}|^2}\sqrt{
\sum_{k}|\Phi_{kj}|^2} }
\eeq
 \cite{DE, Tro}.
 
A parallel development  in image denoising   pioneered by Osher and coworkers \cite{RO, ROF} seeks to enforce  edge detection by  total variation
minimization (TV-min)
\beq
\label{tv0}
\min \int |\nabla g| \quad \text{s.t.} \quad  \int |g-f |^2 \leq \ep^2
\eeq
where $f$ is the noisy image and $\ep$ is the noise level. 
The idea is that for the class of piecewise constant functions, the gradient is sparse and
can be effectively enforced by TV-minimization. 

For digital images, TV-min approach to deblurring  can be formulated
as follows. Let ${ f} \in \IC^{p\times q}$ be  a noisy
complex-valued data of $p\times q$ pixels. 
Let $T$ be the transformation from the true object 
to the ideal sensors, modeling the
imaging process. 
Replacing the total variation in (\ref{tv0})  by  the discrete total variation
\beqn
\|g\|_{\rm TV}&\equiv &\sum_{i,j} \sqrt{|\Delta_1 g(i,j)|^2+ |\Delta_2 g(i,j)|^2},\\
\Delta g= (\Delta_1 {g}, \Delta_2 {g})(i,j)&\equiv &(g(i+1,j)-g(i,j), g(i,j+1)-g(i,j))\label{dg}
\eeqn
 we obtain 
\beq
\label{tv1}
\min \|g\|_{\rm TV} \quad \text{s.t.} \quad  \|T { g}-{f} \|_2\leq \ep 
\eeq
 cf. \cite{CL, CS}. 

In a breakthrough paper \cite{CRT}, Cand\`es {\em et al.} 
show the equivalence of (\ref{tv1}) to (\ref{L1})  for  a random partial Fourier matrix with noiseless data ($\ep=0$)
 and obtain a performance guarantee of 
{exact } reconstruction of {piece-wise constant} objects
from (\ref{tv1}). 

A main application of this present work is to extend the result of \cite{CRT} to inverse scattering with noisy data. 
In this context
it is natural to work with the continuum setting in which
the object is a vector in an infinite dimensional function space,
e.g. $L^2(\IR^d)$. To fit into the CS's discrete framework, we
discretize the object function by pixelating the ambient space
with a regular grid of equal spacing  $\ell$.

 The grid spacing $\ell$  can be thought of as the resolution length, the fundamental parameter  of the discrete model from which all
other parameters are derived. 
For example, the total number of resolution cells is proportional to
$\ell^{-d}$, i.e. $m=\cO(\ell^{-d})$.  As we will assume
that the original object is well approximated by
the discrete model in the limit $\ell\to 0$, 
the sparsity $s$ of the edges of a piecewise constant object
is proportional to $\ell^{1-d}$, i.e. the object is non-fractal. 
It is important to keep in mind the continuum origin of
the discrete model  in order
to avoid confusion about the small 
$\ell$ limit throughout the paper.

{
 First we introduce the notation for multi-vectors $\bY\in \IC^{n\times d}$ 
\beq
\|\bY\|_{b,a}&=&\Big(\sum_{j=1}^n  \|\hbox{\rm row}_j (\bY)\|^b_a\Big)^{1/b},\quad a, b\geq 1\label{90}
\eeq
where $\hbox{\rm row}_j(\bY)$ is the $j$th { row } of $\bY$. 
The 2,2-norm  is exactly the Frobenius norm. To avoid  confusion with the {\em subordinate} matrix norm
\cite{GV},  it is more convenient to view $\bY$ as multi-vectors rather
than a matrix. 

 We aim at  the following  error
 bounds. 
 Let $V$ be the discretized  object 
 and $\hat V$ an estimate of $V$. We will propose
 a compressive sampling scheme that 
 leads to the error bound for the TV-minimizer $\hat V$
 \beq
 \label{v}
 \|\Delta V-\Delta\hat V\|_{2,2} =\cO(\ep),\quad \ell \to 0
 \eeq
implying 
via the discrete Poincare inequality that
 \beq
 \label{vg'}
 \|V-\hat V\|_{2} =\cO({\ep\over {\ell}}) 
 \eeq
{\em  independent of the ambient dimension}.

 If
 $\hat V$ is the reconstruction by using  a  version of the greedy algorithm, Orthogonal Matching Pursuit (OMP)  \cite{DMA, PRK}, for multi-vectors
 then in addition to (\ref{v}) we also have
 \beq
 \label{vg}
 \|V-\hat V\|_{2} =\cO({\ep\over \sqrt{\ell}}) 
 \eeq
{\em  independent of the ambient dimension} (Section \ref{sec:omp}).  We do not know if 
the bound (\ref{vg}) applies to the TV-minimizer.

A key advantage of the greedy algorithm used to prove (\ref{vg})  is the exact recovery of
the gradient support (i.e. the edge location) under
proper conditions  (Theorem \ref{thm4}, Section \ref{sec:omp}). On the one hand,  TV-min requires
fewer data for recovery: $\cO(s) $ for TV-min under RIP
versus $\cO(s^2)$ for 
the greedy algorithm under incoherence where
 the sparsity $s=\cO(\ell^{1-d})$  as already mentioned. 
 On the other hand, the greedy algorithm is computationally
more efficient and incoherent measurements  are much easier to design and verify
than RIP.

\commentout{
In \cite{cis-siso} we propose  and analyze inverse scattering method capable of compressive imaging.  The method employs  randomly and repeatedly (multiple-shot)  the single-input-single-output (SISO)
measurements capable of 
reconstructing  objects of  sufficiently low sparsity in appropriate bases. 

 In this brief note, we extend the  result of \cite{cis-siso} to
 the case of piecewise smooth objects and give a compressed sensing analysis and performance guarantee for the total variation deblurring/denoising principle applied
 to inverse scattering. 
A key idea of our approach is to transform the total
variation minimization (TVM)  \cite{Cha, CL, CGM, CS, RO, ROF}  in inverse scattering to a new version of Basis Pursuit Denoising (BPDN) for compressed sensing  of constrained
joint sparsity (CJS)}

At heart our theory is based on reformulation of
TV-min as CS of  joint sparsity with linear constraints (such as
curl-free constraint in the case of TV-min): BPDN for constrained joint sparsity (CJS) is formulated as
\beq
\min \|\bZ\|_{1,2},\quad\hbox{s.t.}
\quad \|\bY-\varphi(\bZ)\|_{2,2}\leq \ep,\quad \cL\bZ=0 
\label{bp1}
\eeq
where 
\[
\varphi(\bZ)=\lt[\bA_1 Z_1,\ldots, \bA_d Z_d\rt],\quad Z_j=\hbox{the $j$th column of $\bZ$}
\]
where $\cL$ represents a linear constraint. 
Without loss of generality, we assume the matrices $\{\bA_j\}\subset \IC^{n\times m}$  all have {\em unit 2-norm} columns.

In connection to TV-min, $Z_j$ is the $j$-th directional gradient of the discrete object $V$. And from the definition of discrete gradients, it is clear 
that every measurement of $Z_j$ can be deduced from 
two measurements of the object $V$, slightly shifted in
the $j$-th direction  with respect to each other.
 As shown below, for inverse scattering  $\bA_j=\bA,\forall j$ and 
$\cL$ is the curl-free constraint which takes the form 
\[
\Delta_1 Z_2=\Delta_2 Z_1
\]
 for $d=2$ (cf.  (\ref{curl})). 
Our main results, Theorem \ref{thm:bp2} and Theorem \ref{thm4}, constitute performance guarantees for CJS
based, respectively,  on RIP and incoherence of the measurement matrices $\bA_j$. 

\subsection{Comparison of existing theories}
The gradient-based method of \cite{PMG} 
modifies the original Fourier measurements to obtain Fourier measurements of the corresponding vertical and horizontal edge images which then are separately reconstructed by the standard CS algorithms. This approach attempts to take  advantage of usually lower {\em separate}
sparsity and is different from TV-min.  Nevertheless, a similar 2-norm error bound (Proposition V.2, \cite{PMG}) to
(\ref{vg'}) is obtained. 

 Needell and Ward \cite{NW} 
obtain interesting  results for the {\em anisotropic} total variation (ATV) minimization  
in terms of the objective function 
\[
\|g\|_{\rm ATV}\equiv \sum_{i,j} |\Delta_1 g(i,j)|+ |\Delta_2 g(i,j)|.
 \]
 While for {\em real}-valued objects in two dimensions,
  the isotropic TV semi-norm is equivalent to
  the anisotropic version, the two semi-norms are, however,  not the same
  in dimension $\geq 3$ and/or 
  for complex-valued objects. A rather remarkable result of  \cite{NW}
  is the bound $ \|V-\hat V\|_{2} =\cO({\ep})$, modulo
  a logarithmic factor, for $d=2$.  
This is achieved by  proving  a strong Sobolev inequality for two dimensions under the additional assumption of
  RIP with respect to the bivariate Haar transform.
  Unfortunately, 
  this latter assumption prevents the results in 
  \cite{NW} from being directly applicable to structured measurement matrices such as Fourier-like matrices  
which typically
 have high mutual coherence with any compactly supported wavelet
 basis when adjacent  subbands are present.  
Their approach also does not guarantee  exact recovery of the gradient support. 

It is worthwhile to further consider 
these existing approaches from the perspective of the CJS framework
for 
arbitrary $d$. 
The approach of \cite{PMG} can be reformulated as
 solving  $d$ standard BPDN's
\[
\min \|Z_\tau\|_{1},\quad\hbox{s.t.}
\quad \|Y_\tau-\bA Z_\tau\|_{2}\leq \ep, \quad \tau=1,\ldots, d. 
\]
 {\em separately without}  the curl-free constraint $\cL$ where
 $Z_\tau$ and $Y_\tau$ are, respectively,  the $\tau$-th columns of $\bZ$ and $\bY$. To recover the original image from the directional gradients,
 an additional step of consistent integration  becomes an important part of the approach
 in \cite{PMG}.  
 
 From the CJS perspective, the ATV-min considered
 in \cite{NW} can be reformulated  as follows. 
 Let $\tilde Z\in \IC^{dm}$  be the image gradient vector by stacking the $d$ directional gradients   and let $\tilde Y\in \IC^{dn}$ be the similarly concatenated data vector.
 Likewise let $\tilde\bA=\hbox{diag}(\bA_1,\ldots,\bA_d)\in \IC^{dn\times dm} $ be the 
 block-diagonal matrix with blocks $\bA_j\in \IC^{n\times m}$. Then ATV-min is equivalent to
 BPDN for a {\em single} constrained and concatenated  vector 
\beq
\min \|\tilde Z\|_{1},\quad\hbox{s.t.}
\quad \|\tilde Y-\tilde\bA\tilde Z\|_{2}\leq \ep,\quad \tilde\cL\tilde Z=0.
\label{bp0}
\eeq
where $\tilde \cL$ is the same constrain $\cL$ reformulated
for concatenated vectors. 
Repeating {\em verbatim} the proofs of
Theorems \ref{thm:bp2} and \ref{thm4} we obtain
the same error bounds as (\ref{v})-(\ref{vg}) for ATV-min as formulated in  (\ref{bp0}) under the same conditions
for $\bA_j$ separately. 

ATV-min is  formulated
differently  in \cite{NW}. Instead of image gradient,  it is formulated 
in terms of the image to get rid of the curl-free constraint. To  proceed 
the differently  concatenated  matrix $[\bA_1,\ldots,\bA_d]$ is then assumed  to satisfy  
RIP of higher order   
demanding $2dn$ measurement data.  
For $d=2$, RIP of order $5s$ with $\delta_{5s}<1/3$
is assumed for $[\bA_1,\bA_2]$ in \cite{NW} which is  much
more stringent than RIP of order $2s$
with 
$\delta_{2s}<\sqrt{2}-1$ for $\bA_1,\bA_2$ {\em separately} in (\ref{bp0}). In particular, $\bA_1=\bA_2$  is allowed for 
(\ref{bp0}) but not for \cite{NW}. 
To get the  favorable $\cO(\ep)$ 2-norm error bound for $d=2$, additional 
measurement matrix satisfying RIP with respect to
the bivariate Haar basis is needed, which, as mentioned 
above,  excludes partial Fourier measurements. 

\subsection{Organization}
The rest of the paper is organized as follows.  In Section \ref{sec5}, we present a performance guarantee
  for BPDN for CJS and obtain
  error bounds.
  In Section \ref{sec:omp}, we analyze the greedy approach
  to sparse recovery of CJS and derive error bounds,
  including 
 an improved 2-norm error bound. In Section \ref{sec2}, we review
  the scattering problem starting from the continuum setting
  and introduce the discrete model.  In Section \ref{sec3},
  we discuss various sampling schemes including
  the forward and backward sampling schemes for inverse scattering for point objects. In Section \ref{sec:tv}
  we formulate TV-min for piecewise constant objects
  as BPDN for CJS. 
  We present numerical examples and conclude in Section \ref{sec7}. 
 We present the proofs in the Appendices. 

\section{BPDN for CJS}\label{sec5}

\commentout{
For $\bZ\in \IC^{m\times p}$ define the notation 
\beq
\|\bZ\|_{b,a}&=&\Big(\sum_{j=1}^m  \|\hbox{\rm row}_j (\bZ)\|^b_a\Big)^{1/b},\quad a, b\geq 1\label{90}
\eeq
where $\hbox{\rm row}_j(\bZ)$ is the $j$th { row } of $\bZ$. 
{
The 2,2-norm  is exactly the Frobenius norm 
of jointly sparse multi-vectors as a  matrix.} The notation in (\ref{90}) should not be confused with the {\em subordinate} matrix norm
\cite{GV}. 
}

Consider the linear inversion problem 
\beq
\label{200}
\bY=\varphi(\bX)+\bE,\quad
\cL\bX=0
\eeq
where 
\[
 \varphi(\bX)=[\bA_1X_1, \bA_2X_2,\ldots,\bA_dX_d],\quad \bA_j\in \IC^{n\times m}
 \] and the corresponding 
 BPDN 
\beq
\min \|\bZ\|_{1,2},\quad\hbox{s.t.}
\quad \|\bY-\varphi(\bZ)\|_{2,2}\leq \ep=\|\bE\|_{2,2},\quad \cL\bZ=0.
\label{bp2}
\eeq
 For TV-min in $d$ dimensions, $\bA_j=\bA,\forall j$, $\bX$ represents
 the discrete gradient of the unknown object $V$ and 
$\cL$ is the curl-free constraint. 
Without loss of generality, we assume the matrices $\{\bA_j\}$
all have unit 2-norm columns.

{
 We say that $\bX$ is $s$-row sparse if the number of
 nonzero rows in $\bX$ is at most $s$. With a slight abuse of terminology we call
 $\bX$ the object (of CJS). 
 
 In the following theorems, we let the object $\bX$ be general,  not necessarily $s$-row sparse. Let $\bX^{(s)}$ consist of $s$ largest rows in the 2-norm of $\bX$. Then
$\bX^{(s)}$ is the  best $s$-row sparse approximation of $\bX$. }
 
\begin{theorem}\label{thm:bp2}
Suppose that the linear map $\varphi$ satisfies the RIP of order
$2s$ 
\beq 
\label{rip}
(1-\delta_{2s})\|\bZ\|_{2,2}^2\leq
\|\varphi(\bZ)\|_{2,2}^2\leq (1+\delta_{2s})\|\bZ\|_{2,2}^2
\eeq
for any $2s$-row sparse  $\bZ$
with
\[
\delta_{2s}<\sqrt{2}-1.
\]
 
Let $\hat\bX$ be the minimizer of (\ref{bp2}). Then
 \beq
 \|\hat \bX-\bX\|_{2,2}&\leq & C_1s^{-1/2}\|\bX-\bX^{(s)}\|_{1,2}+C_2\ep\label{70}
 \eeq
 for absolute  constants $C_1, C_2$ depending only on $\delta_{2s}$.  

\end{theorem}
\begin{remark}
Note that the RIP for joint sparsity (\ref{rip}) 
follows straightforwardly from the assumption of separate RIP
\beqn 
(1-\delta_{2s})\|Z\|_{2}^2\leq
\|\bA_jZ\|_{2}^2\leq (1+\delta_{2s})\|Z\|_{2}^2,\quad \forall j
\eeqn
with  a common RIC. 
\end{remark}
\begin{remark}
For the standard Lasso with a particular choice
of regularization parameter, Theorem 1.3 of  \cite{CP} guarantees 
exact support recovery under a favorable sparsity constraint. 
In our setting and notation, their TV-min principle corresponds to 
\beq
\min_{\cL\bZ=0} \lambda \sigma \|\bZ\|_{1,2}+{1\over 2}
\|\bY-\varphi(\bZ)\|_{2,2}^2,\quad \lambda=2\sqrt{2\log m}
\label{lasso}
\eeq
where $\sigma^2=\ep^2/(2n)$ is the variance of
the assumed Gaussian  noise in each entry of $\bY$.
 Unfortunately, even  if the result of \cite{CP} can
be extended to (\ref{lasso}), it 
is inadequate for our purpose because \cite{CP} assumes 
independently selected support and signs, which
is clearly not satisfied by the gradient of
a piecewise constant object. 
\end{remark}

The proof of Theorem \ref{thm:bp2}  is given in Appendix A. 

The error bound (\ref{70}) implies (\ref{v})
for $s$-row sparse $\bX$. For the 2-norm bound
(\ref{vg'}), we apply
 the discrete Poincare inequality \cite{Che}
\[
\|f\|_2^2\leq {m^{2/d}\over 4d} \|\Delta f\|_2^2
\]
to get 
\beq
\label{80}
\|V-\hat V\|_2\leq {m^{1/d}\over 2{d}^{1/2}} C_2\ep=\cO({\ep\over \ell}).
\eeq

\section{Greedy pursuit for CJS}
\label{sec:omp}

One idea to improve the error bound  is through 
exact recovery of the support.
This can be achieved by greedy algorithms. As before, we consider the general linear inversion
with CJS  (\ref{200}) with
$\|\bE\|_{2,2}=\ep$. 

\bigskip

Our following algorithm is an extension of the joint-sparsity greedy algorithms  of
\cite{MMV05,CH06,Tropp} to the setting with  multiple sensing matrices. 
\begin{center}
   \begin{tabular}[width=5in]{l}
   \hline
   \centerline{{\bf Algorithm 1.}\quad  OMP for joint sparsity} \\ \hline
   Input: $\{\bA_j\}, \bb,\eta>0$\\
 Initialization:  $\mbx^0 = 0, \bR^0 = \bb$ and $\cS^0=\emptyset$ \\ 
Iteration:\\
\quad 1) $i_{\rm max} = \hbox{arg}\max_{i}\sum^d_{j=1}|\Phi^*_{j,i}R^{k-1}_j |,\hbox{where $\Phi^*_{j,i}$ is the conjugate transpose  of $i$-th column of $\bPhi_j$} $\\
 \quad      2) $\cS^k= \cS^{k-1} \cup \{i_{\rm max}\}$ \\
  \quad  3) $\mbx^k = \hbox{arg} \min\|
     \bA \bZ-\bY\|_{2,2}$ s.t. \hbox{supp}($\bZ$) $\subseteq S^k$ \\
  \quad   4) $\bR^k = \bb- \varphi(\mbx^k)$\\
\quad  5)  Stop if $\sum_j\|R^k_j\|_{2}\leq \ep$.\\
 Output: $\bX^k$. \\
 \hline
   \end{tabular}
\end{center}
\bigskip

Note that the linear constraint is not enforced in Algorithm 1.

A natural indicator of the performance of OMP is
the mutual coherence (\ref{mu}) \cite{Tro, DET}.  
Let
\[
\mu_{\rm max}=\max_{j}\mu(\bA_j). 
\]
Then analogous to Theorem 5.1 of \cite{DET}, we have
the following performance guarantee.
\begin{theorem}  Suppose the sparsity $s$ satisfies 
\beq
\label{sparse}
s<{1\over 2} (1+{1\over \mu_{\rm max}})-{\sqrt{d}\ep\over \mu_{\rm max} X_{\rm min}},\quad X_{\rm min}=\min_{k} \|\hbox{\rm row}_k (\bX)\|_1.
\eeq
 Let $\bZ$ be the output of Algorithm 1, with the stopping rule that the residual drops to the level $\ep$ or below.  Then  $\supp (\bZ)=\supp (\bX)$.

 Let $\hat \bX$ solve
the least squares problem
\beq
\label{ls2}
\hat\bX=\hbox{\rm arg} \,\min_{\bB}\|\bY-\bPhi\bB\|_{2,2},\quad \hbox{s.t.}\quad  \supp (\bB)\subseteq \supp (\bX),\quad \cL\bB=0.
\eeq
Then
\beq
\label{err}\|\hat \bX-\bX\|_{2,2}
\leq {2\ep \over \sqrt{1-\mu_{\rm max}(s-1)}}.
\eeq
\label{thm4}
\end{theorem}
\commentout{
\begin{remark}
The least squares solution $\hat \bX$ is not necessarily 
the discrete gradient of the following
least squares solution on the level of object itself
\beq
\label{ls3}
\hat X=\hbox{\rm arg} \,\min_{Z}\|Y-\bPhi Z\|_{2}
\eeq
where the gradient of $Z$ is supported on $\supp(\bX)$, since
the objective functions are different.
\end{remark}
\begin{remark}
In the case of a gradient object $\bX$ 
other  practical methods
of reconstructing the original object from the recovered
gradient $\hat\bX$ are discussed in \cite{PMG}.
\end{remark}
}

The proof of Theorem \ref{thm4} is given in Appendix B.

The main advantage of Theorem \ref{thm4} over Theorem \ref{thm:bp2} is the guarantee of exact recovery of the support of $\bX$. Moreover,  a better 
2-norm error bound follows because now
the gradient error is guaranteed  to vanish outside a set
of cardinality  $\cO(\ell^{1-d})$: 
Let $\ell\IL_l\subseteq \ell\IL, l=1,..., L$ be the 
level sets of the object $V$ such that  
\[
V=\sum_{l=1}^L v_l \ind_{\ell\IL_l}
\]
where $\IL_l\cap\IL_k=\emptyset, l\neq k, \IL=\cup_{l}\IL_l$. 
The reconstructed object $\hat V$ from  $\hat\bX$
given in (\ref{ls2})
also takes the same form
\[
\hat V=\sum_{l=1}^L \hat v_l \ind_{\ell\IL_l}.
\]
To fix the undetermined constant, we may assume that $v_1=\hat v_1$. Since 
\[
\|\Delta (V-\hat V)\|_{2,2}=\cO(\ep)
\]
 by (\ref{err}) and the gradient error occurs only on 
the boundaries of $\ell\IL_l$ of cardinality  $\cO(\ell^{1-d})$, we have 
\[
|v_l-\hat v_l|=\cO(\ep\ell^{(d-1)/2}),\quad\forall l. 
\]
Namely
\beqn
\|V-\hat V\|_\infty=\cO(\ep\ell^{(d-1)/2})
\eeqn
and thus
\beqn
\label{71}
\|V-\hat V\|_2=\cO({\ep\over\sqrt{\ell}}). 
\eeqn

\section{Application: inverse scattering}\label{sec2}
In this section, we discuss the main application of
the CJS formulation, i.e.  the TV-min for  inverse scattering problem.

  A monochromatic  wave $u$
propagating in a heterogeneous medium characterized by
a variable refractive index  $n^2(\br)=1+v(\br)$ is governed by the  Helmholtz
  equation 
  \beq
  \label{helm}
  \nabla^2 u(\br)+\om^2(1+v(\br)) u(\br)=0
  \eeq 
  where $v$ describes the medium inhomogeneities. 
 For simplicity,   the wave velocity is assumed to be unity and 
 hence the wavenumber $\om$ equals the frequency.

   Consider the scattering of the incident plane wave
 \beq
 u^{\rm i}(\br)=e^{i\om \br\cdot\bdhat}
 \eeq
 where $\bdhat$ is the incident direction. 
The scattered field $u^{\rm s}=u-u^{\rm i}$
then satisfies 
\beq
\label{scattered}
\nabla^2 u^{\rm s}+\om^2 u^{\rm s}=-\om^2v u 
\eeq
which can be written as the Lippmann-Schwinger equation:
\beq
\label{exact'}
u^{\rm s}(\br)&=&\om^2\int_{\IR^d} v(\br') 
\lt(u^{\rm i}(\br')+u^{\rm s}(\br')\rt) G(\br, \br')d\br' 
\eeq
where $G$ is  the Green function for the operator $-(\nabla^2+\om^2)$. 

The scattered field necessarily satisfies Sommerfeld's radiation condition 
\[
\lim_{r\to\infty} r^{(d-1)/2} \Big({\partial \over \partial r} -i\omega\Big) u^{\rm s}=0
\]
reflecting the fact that the energy which is radiated from the sources represented by the right hand side of (\ref{scattered}) must scatter to infinity. 

Thus the scattered field has
the far-field asymptotic
\beq
u^{\rm s}(\br)={e^{i\om |\br|}\over |\br|^{(d-1)/2}}\lt(A
(\hat\br,\bdhat,\om)+\cO(|\br|^{-1})\rt),\quad \hat\br=\br/|\br|,
\eeq
where $A$ is the scattering amplitude and $d$ the spatial dimension. 
In inverse scattering theory,  
 the  scattering amplitude is the measurement data
 determined by the formula  \cite{CK}
\beqn
\label{sa}
A(\hat\br,\bdhat,\om)&=&{\om^2\over 4\pi}
\int d\br' v(\br') u(\br') e^{-i\om \br'\cdot\hat\br}
\eeqn
which under the Born approximation becomes
\beq
\label{sa2}
A(\hat\br,\bdhat,\om)&=&{\om^2\over 4\pi}
\int d\br' v(\br') e^{i\om \br'\cdot(\bdhat-\hat\br)}
\eeq

For the simplicity of notation we consider 
the  two dimensional case in detail. 
Let $\ID\subset \IZ^2$ be a square sublattice of 
$m$ integral points. Suppose that $s$ point scatterers  are located in 
 a square lattice of spacing $\ell$
\beq
\label{latt}
\ell\ID=\lt\{\br_j=\ell(p_1,p_2):  j=(p_1-1)\sqrt{m}+p_2, \bp=(p_1,p_2)\in \ID\rt\}. 
\eeq
In the context of inverse scattering, it is natural to treat  
the size of the discrete ambient domain $\ell\IL$ being
fixed independent of the resolution length $\ell$. In particular, 
$m\sim \ell^{-2}$ in two dimensions. 

First  let us motivate the inverse scattering sampling scheme
in the case of   {\em point}  scatterers  and 
 let  $v_{j}, j=1,...,m$ be the strength of
the scatterers. In other words, the total object is a sum of $\delta$-functions
\beq
\label{point}
v(\br)=\sum_{j} v_j \delta(\br-\br_j).
\eeq 
Let  $
\cS=\lt\{\br_{i_j}: j=1,...,s\rt\}$
be the
locations of the scatterers. Hence $v_j=0, \forall \br_j\not \in \cS$.

For point objects the scattering amplitude
becomes a finite sum 
\beq
\label{sapt}
A(\hat\br,\bdhat,\om)&=&{\om^2\over 4\pi}
\sum_{j=1}^m v_j 
 e^{i\om \br_j\cdot(\bdhat-\hat\br)}. 
\eeq
In  the Born approximation  the exciting
   field $u(\br_{j})$  is replaced 
   by the incident field $u^{\rm i}(\br_{j})$.  
   
 \section{Sampling schemes}\label{sec3}
 Next we review the sampling schemes introduced in \cite{cis-siso} for {\em point} objects (\ref{point}). 
 
Let $\bdhat_l, \hat\br_l, l=1,...,n$ be various incident and sampling directions for the frequencies $\om_l, l=1,...,n$ to be determined later.
Define  the measurement vector $Y=(y_l)\in \IC^n$ 
with
\beq
\label{4'}
y_l={4\pi\over \om^2\sqrt{n}}A(\hat\br_l,\bdhat_l,\om_l),\quad l=1,...,n. 
\eeq
  The measurement vector is related to the {\em point} object 
vector $X=(v_j)\in \IC^m$  by the sensing matrix  $\bA$  as   
  \beq
Y=\bA X +E
\label{66'}
\eeq
where $E$ is the measurement error. 
  Let  $\theta_l,\tilde\theta_l$ be the polar angles of
  $\bdhat_l,\hat\br_l$, respectively. 
  The $(l,j)$-entry of $\bA\in \IC^{n\times m}$ is
   \beq
   \label{entry}
n^{-1/2}e^{-i\om_l\hat\br_l\cdot\br_j}e^{i\om_l\bdhat_l\cdot \br_j}&=&n^{-1/2}
e^{i\om_l\ell (p_2(\sin{\theta_l}-\sin{\thetatil_l})+p_1(\cos{ \theta_l }-\cos{\thetatil_l}))},\quad j=(p_1-1)+p_2. 
\eeq
Note that  $\bA$ has unit 2-norm columns. 


Let $(\xi_l,\zeta_l)$ be i.i.d. uniform random variables on $[-1,1]^2$ and let  $\rho_l, \phi_l$ be the polar coordinates  as 
in 
\beq
(\xi_l,\zeta_l)=\rho_l (\cos\phi_l,\sin\phi_l),\quad 
\rho_l=\sqrt{\xi_l^2+\zeta_l^2} \leq\sqrt{2}
\eeq
  Let the sampling angle $\tilde\theta_l$ be related
to the incident  angle $\theta_l$ via 
\beq
\label{ch1}
\theta_l+\tilde\theta_l=2\phi_l+\pi,
\eeq
and set the frequency $\om_l$ to be
\beq
\label{ch2}
\om_l= {\Omega \rho_l\over \sqrt{2} \sin{\theta_l-\tilde\theta_l\over 2}}
\eeq
where $\Om$ is a control parameter. 
Then the entries (\ref{entry}) of the sensing matrix $\bA$ 
 under the condition  
\beq
\label{75}
\Omega\ell=\pi/\sqrt{2}
\eeq
are those of  random partial Fourier matrix
\beq
e^{i\pi(p_1\xi_l+p_2\zeta_l)}, \quad
l=1,...,n,\quad p_1, p_2=1,...,\sqrt{m}.\label{17-2}
\eeq

We consider two particular sampling schemes: 
The first one employs multiple frequencies 
with the sampling angle always in the back-scattering direction
resembling the imaging geometry of synthetic aperture radar; the second 
employs only single high frequency with the sampling angle
in the forward
 direction,
resembling  the imaging geometry  of X-ray tomography.

\bigskip

\noindent{\bf I. Backward Sampling} 
This scheme employs  $\Om-$band limited probes, i.e.
$\om_l\in  [-\Om,\Om]$. 
This and (\ref{ch2}) lead to
the constraint:
\beq
\label{const}
{\lt| \sin{\theta_l-\tilde\theta_l\over 2}\rt|}\geq {\rho_l\over\sqrt{2}}. 
\eeq

A simple way  to satisfy (\ref{ch1}) and (\ref{const}) is
to set 
\beq
\label{41}\phi_l&=&\tilde\theta_l=\theta_l-\pi,\\
\label{42} \om_l&=& {\Om\rho_l\over \sqrt{2}}
\eeq
$l=1,...,n$. In this case the scattering amplitude is sampled exactly in 
the backward direction, resembling 
SAR imaging.
 In contrast, the exact forward sampling
with $\tilde\theta_l=\theta_l$ almost surely violates
the constraint  (\ref{const}). 

\bigskip

\noindent {\bf II. Forward Sampling}
 This scheme employs  single  frequency probes no less
 than $\Omega$: 
\beq
\label{20}
\om_l=\gamma\Omega,\quad \gamma\geq 1,\quad   l=1,...,n.\eeq
We set 
\beq
\label{20'}
\theta_l=\phi_l +\arcsin{\rho_l\over \gamma\sqrt{2}}\\
\tilde\theta_l=\phi_l-\arcsin{\rho_l\over \gamma\sqrt{2}}.\label{20''}
\eeq
The difference between the incident angle and
the sampling angle is 
\beq
\label{25}
\theta_l-\tilde\theta_l=2\arcsin{\rho_l\over \gamma\sqrt{2}}
\eeq
which diminishes as $\gamma\to\infty$. In other words, in the high frequency limit, the sampling angle approaches the
incident  angle, resembling  X-ray tomography \cite{Nat}.

\section{Piecewise constant objects}
 \label{sec:tv}
 Next let us consider the following class of  piecewise constant objects:
 \beq
v(\br)&=&\sum_{\bp\in \ID}v_\bp\ind_{\boxdot}({\br\over \ell}-\bp),\quad \boxdot=\lt[-{1\over 2}, {1\over 2}\rt]^2 \label{31'}
\eeq
where $\ind_{\boxdot}$ is the indicator function of the unit square $\boxdot$. As remarked in the Introduction,
we think of  the pixelated $v$ as discrete approximation
of some compactly support function on $\IR^2$ and having
a well-defined limit as $\ell\to 0$. Set $V=(v_j) \in \IC^m, j=(p_1-1)\sqrt{m}+p_2$.

\commentout{
Let  the discrete total variation $\|V\|_{\rm TV}$ be  defined by
\beq
\label{tvm}
\| V\|_{\rm TV}=\sum_{\bp\in \ID} \sqrt{|v_{\bp+\be_1}-v_{\bp}|^2+|v_{\bp+\be_2}-v_{\bp}|^2},\quad\be_1=(1,0),\quad \be_2=(0,1).
\eeq
}

The discrete version of (\ref{sa2}) is, however, not exactly the same as (\ref{sapt}) since extended objects have different scattering properties from those of point objects.

The integral on the right hand side of (\ref{sa2}), modulo the discretization error, is 
\beqn
\int d\br' v(\br') e^{i\om \br'\cdot(\bdhat-\hat\br)}
&=&\sum_{\bp\in \ID} v_\bp e^{\mi \om\ell \bp\cdot (\bdhat-\hat\br)}  \int e^{\mi \om\br'\cdot (\bdhat-\hat\br)}\ind_\boxdot({\br'\over \ell}) d\br'.
\eeqn
Now letting  $\bdhat_l, \hat\br_l, \om_l, l=1,\cdots, n$
be  selected according to Scheme I or II and substituting them in the above equation, we obtain
\beqn
\int d\br' v(\br') e^{i\om_l \br'\cdot(\bdhat_l-\hat\br_l)}
&=&\ell^2\sum_{\bp\in \ID} v_\bp  e^{i\pi (p_1\xi_l+p_2\eta_l)}   \int_{\boxdot} e^{i\pi(x\xi_l+y\eta_l)}dx dy\\
&=&\ell^2\sum_{\bp\in \ID} v_\bp e^{i\pi (p_1\xi_l+p_2\eta_l)}   {2\sin{(\pi \xi_l/2)}\over \pi \xi_l} {2\sin{(\pi \eta_l/2)}\over \pi \eta_l}.
\eeqn

Let 
\[
x_j=\ell^2v_\bp,\quad j=(p_1-1)\sqrt{m}+p_2
\]
and
\[
y_l={4\pi\over \om_l^2 \tilde g_l\sqrt{n}} A(\hat\br_l,\bdhat_l,\om_l)+E_l,\quad l=1,\cdots, n
\]
where
\[
\tilde  g_l=  {2\sin{(\pi \xi_l/2)}\over \pi \xi_l} {2\sin{(\pi \eta_l/2)}\over \pi \eta_l}
\]
where $E=(e_l)$ is the noise vector. 

Define  the sensing matrix $\bA=[\phi_{kp}]$ as
\beq
\phi_{kp}={1\over \sqrt{n}} e^{\mi \pi (p_1\xi_k+p_2\eta_k)},\quad  p=(p_1-1)\sqrt{m}+p_2,\quad  p_1, p_2=1,...,\sqrt{m}. \label{65}
\eeq
Then (\ref{sa}) can be written  in the same form as (\ref{66'})
\beq
\label{66}
Y=\bA X +E,\quad X=(x_j)
\eeq 
where the data and error vectors have been modified as above to account
for the differences between extended and point objects.

Our goal is to establish the performance guarantee for
TV-min
\beq
\label{101}
\min \|Z\|_{\rm TV}
,\quad\hbox{subject to}\quad  \|Y-\bA Z\|_2\leq \|E\|_2.
\eeq
And we accomplish this by transforming (\ref{101}) into  BPDN for CJS (\ref{bp2}).

Define $\bX=(X_1, X_2)$ with
\[
(X_1,X_2)=\ell^2(\Delta_1V , \Delta_2V)
\in \IC^{m\times 2}. 
\]
 Suppose the support of $\{v_{\bp+\be_1}, v_{\bp+\be_2} \}$ is contained in $\ID$. Simple calculation yields  that
\beqn
y_l&=&  {\ell^2\over \sqrt{n}} e^{i\pi\xi_l}\sum_{\bp\in \ID} v_{\bp+\be_1}e^{\mi\pi (p_1\xi_l+p_2\eta_l)}\\
&=&  {\ell^2\over \sqrt{n}} e^{i\pi\eta_l}  \sum_{\bp\in \ID} v_{\bp+\be_2}e^{\mi\pi (p_1\xi_l+p_2\eta_l)}
\eeqn
and thus
\beq
(e^{-\mi\pi\xi_l}-1)y_l&=&  {\ell^2\over \sqrt{n}}  \sum_{\bp\in \ID} (v_{\bp+\be_1}-v_\bp)e^{\mi\pi (p_1\xi_l+p_2\eta_l)}\\
(e^{-\mi\pi\eta_l}-1)y_l&=&  {\ell^2\over \sqrt{n}}  \sum_{\bp\in \ID} (v_{\bp+\be_2}-v_\bp)e^{\mi\pi (p_1\xi_l+p_2\eta_l)}. 
\eeq
Define $\bY=(Y_1, Y_2)$ with
\[
Y_1=\lt( (e^{-\mi\pi\xi_l}-1)y_l\rt),\quad Y_2=\lt(
(e^{-\mi\pi\eta_l}-1)y_l
 \rt) \in \IC^{n}
\]
and $\bE=(E_1, E_2)$ with
\beq
\label{error}
E_1=\lt( (e^{-\mi\pi\xi_l}-1)e_l\rt),\quad E_2=\lt(
(e^{-\mi\pi\eta_l}-1)e_l
 \rt)\in \IC^{n}.
 \eeq
\begin{figure}
 \begin{center}
\includegraphics[width=0.8\textwidth]{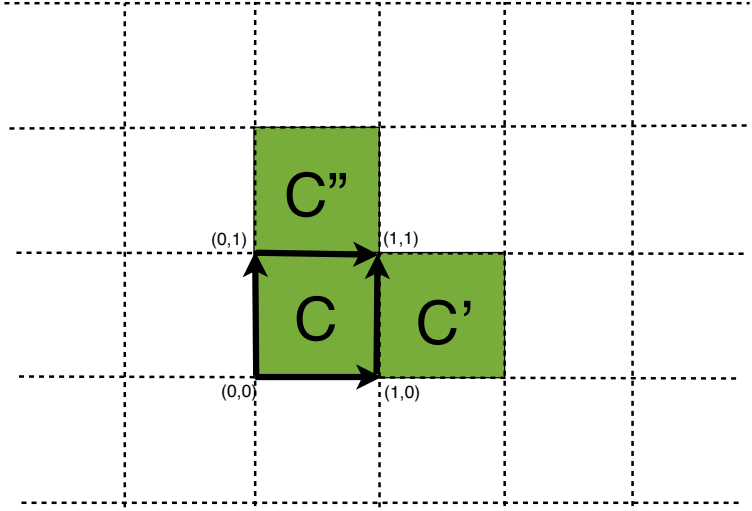}
\end{center}
\caption{Consistency  among cells $C,C'$ and $C''$. } 
\label{fig:grid}
 \end{figure}

We rewrite (\ref{66}) in the form
\beq
\label{67}
\bY=\bA\bX+\bE.
\eeq
subject to the constraint
\beq
\label{curl}\label{c1}
\Delta_1 X_2=\Delta_2 X_1
\eeq
which is the discrete version of curl-free condition.
This ensures that the reconstruction by line integration of $(v_\bp) $
from $\bX$ is consistent (i.e. path-independent).

To see that (\ref{curl}) is necessary and sufficient 
for the recovery  of $(v_\bp) $, consider,  for example, the notations in Figure \ref{fig:grid} and suppose $v_{0,0}$ is known. 
By definition of the difference operators $\Delta_1, \Delta_2$ we have
\beqn
v_{1,0}&=&v_{0,0}+(\Delta_1V)_{0,0}\\
v_{0,1}&=&v_{0,0}+(\Delta_2V)_{0,0}
\eeqn
In general, we can determine $v_\bp, \bp\in \ID$ iteratively
from the relationship
\beqn
v_{\bp+\be_1}&=&v_{\bp}+(\Delta_1V)_{\bp}\\
v_{\bp+\be_2}&=&v_{\bp}+(\Delta_2V)_{\bp}
\eeqn
and the knowledge of  $V$ at any grid point.
The path-independence in evaluating  $v_{p_1+1,p_2+1}$
\beqn
v_{p_1+1,p_2+1}&=v_{p_1,p_2}+(\Delta_1V)_{p_1,p_2}+(\Delta_2V)_{p_1+1,p_2}\\
&=v_{p_1,p_2}+(\Delta_2V)_{p_1,p_2}  +(\Delta_1V)_{p_1, p_2+1}
\eeqn
implies
that 
\[
(\Delta_2V)_{p_1+1,p_2}-(\Delta_2V)_{p_1,p_2} 
=(\Delta_1V)_{p_1, p_2+1}-(\Delta_1V)_{p_1,p_2}
\]
which is equivalent to (\ref{curl}).

Now eq. (\ref{66}) is equivalent to
(\ref{67}) with the constraint (\ref{c1}) provided that
the value of $V$ at (any) one grid point is known.  
The equivalence between the original TV-min (\ref{101})
and the CJS formulation (\ref{bp2}) with $\bA_j=\bA, \forall j$ then hinges
on the equivalence of their respective feasible sets which can be established
under the assumption of Gaussian noise.
When $E$ in (\ref{66}) is Gaussian noise, 
then so is $\bE$ and vice versa, with variances precisely related
to each other.

The random partial Fourier measurement matrix satisfies RIP with $n=\cO(s)$, up to a logarithmic factor \cite{CRT},  while its 
mutual coherence $\mu$ behaves like
$\cO(n^{-1/2})$ \cite{CRS-pt}.   Therefore  (\ref{sparse}) impies  the sparsity constraint $s=\cO(\sqrt{n})$ for the greedy approach which
is more stringent  than $s=\cO(n)$ for the BPDN approach.

\begin{figure}[t]
\begin{center}
\includegraphics[width=0.4\textwidth]{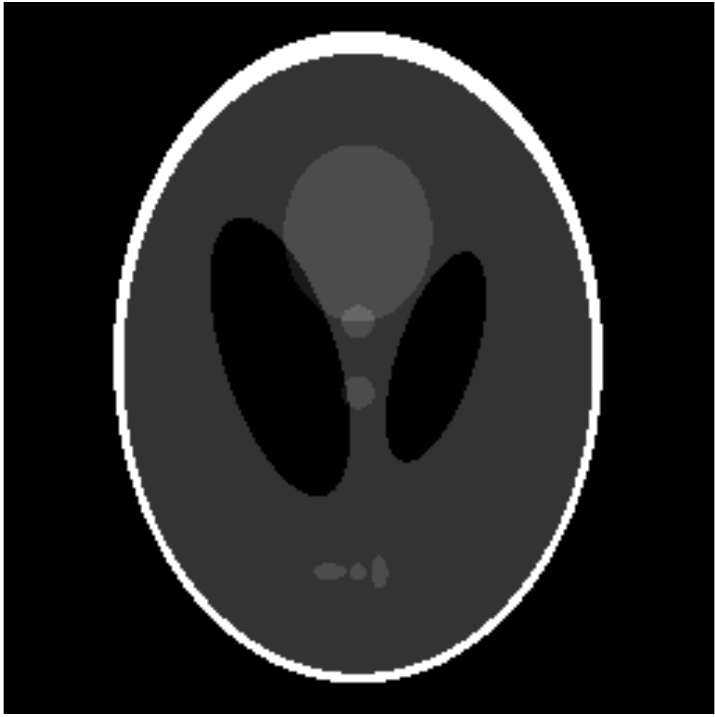}
\includegraphics[width=0.4\textwidth]{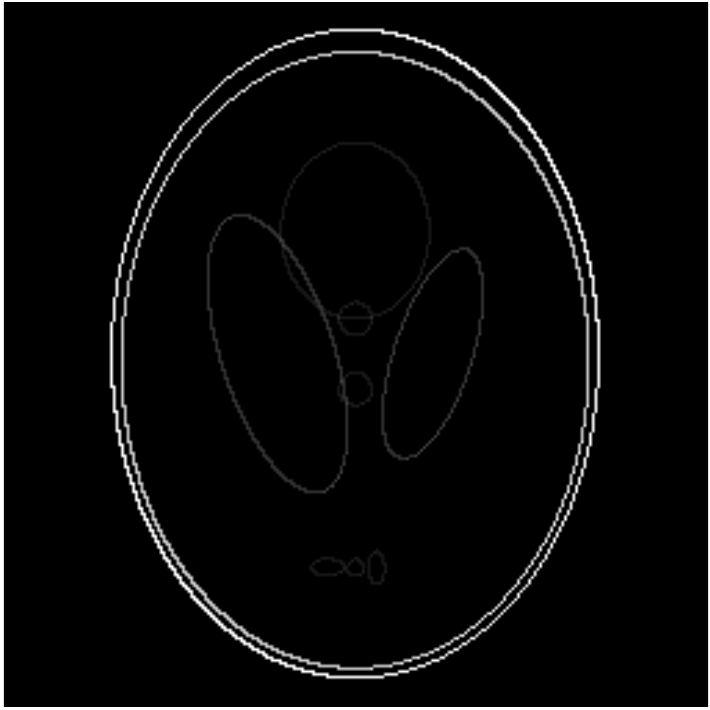}
\end{center}
\caption{The original $256\times 256$ Shepp-Logan phantom  (left), the Shepp-Logan phantom and the magnitudes of its gradient with sparsity $s=2184$.}
\label{fig2}
\end{figure}

\commentout{
To conclude our result about TV-min, let us
quote the following useful estimate for RIC \cite{Rau}. 
\begin{proposition}
Suppose \beq
{n\over \ln{n}}\geq C \delta^{-2}k\ln^2{k} \ln{m} \ln{1\over \alpha},\quad
\alpha\in (0,1)
\eeq
for given sparsity $k$ where $C$ is an absolute constant. 
Then the restricted isometry constant 
of  the matrix (\ref{17-2}) with (\ref{75}) satisfies 
\[
\delta_k\leq \delta
\]
with
probability at least $1-\alpha$.\label{prop2}
\end{proposition}

Using Theorem \ref{thm:bp2} and Proposition \ref{prop2} with 
$\bA_j=\bA$, $k=2s$ and $\delta<\sqrt{2}-1$, 
we obtain the error bound for inverse scattering. 

\begin{theorem} \label{thm:tv}
Let the probe frequencies $\om_l$, the incident angles $\theta_l$
and the sampling angles $\tilde\theta_l$ are
determined by by the Backward or Forward
Sampling scheme. 
 
Suppose 
\beq
\label{75-tv}
\Omega\ell=\pi/\sqrt{2}
\eeq
and suppose 
\beq
\label{77-tv}
{n\over \ln{n}}\geq C_0 \delta^{-2}s\ln^2{s} \ln{m} \ln{1\over \alpha},\quad
\alpha\in (0,1)
\eeq
 holds for some constant $C_0$ and  any $\delta<\sqrt{2}-1$.
Then   the estimate (\ref{70}) holds true 
with probability at least $1-\alpha$.
\end{theorem}
Now suppose the object gradient is $s$-row sparse. Theorem \ref{thm:tv} says that the number of data
 needed for an accurate TV-min solution 
is  $\cO(s)$, modulo logarithmic factors,  which is $\cO(\ell^{-1})$ for piecewise constant
objects in two dimensions.}

\section{Conclusion}\label{sec7}
We have developed a general CS theory (Theorems \ref{thm:bp2} and \ref{thm4}) for
constrained joint sparsity with multiple sensing matrices and obtained performance guarantees
parallel to those for the CS theory for single measurement vector and matrix. 

From the general theory 
we have derived 2-norm error bounds for the object and the gradient,
independent of the ambient dimension, 
for  TV-min 
and greedy estimates  of piecewise constant objects. 
 
In addition,  the CJS greedy algorithm can recover exactly the gradient support (i.e. the edges of the object) leading to
an improved 2-norm  error bound.  Although the CJS greedy algorithm needs a higher number of measurement data than
TV-min for Fourier measurements
the  incoherence property required is much easier to check and often the only practical way to verify RIP when the measurement matrix is not i.i.d. or Fourier.

 
We end by presenting a numerical example
 demonstrating the noise stability of the TV-min. 
Efficient algorithms for TV-min denoising/deblurring 
 exist \cite{BT, WBA}. We use the open source code {\em L1-MAGIC} ({\tt http://users.ece.gatech.edu/\~\,justin/l1magic/}) for our simulation.

Figure \ref{fig2} shows the $256\times 256$  image
of the Shepp-Logan Phantom (left) and the modulus of its gradient (right). Clearly the sparsity ($s=2184$) of
the gradient is much smaller than that of the original image.
We take $10000$ Fourier measurement data for
the L1-min  (\ref{L1})  and TV-min (\ref{tv1})  reconstructions.

 \begin{figure}[t]
\begin{center}
\includegraphics[width=0.3\textwidth]{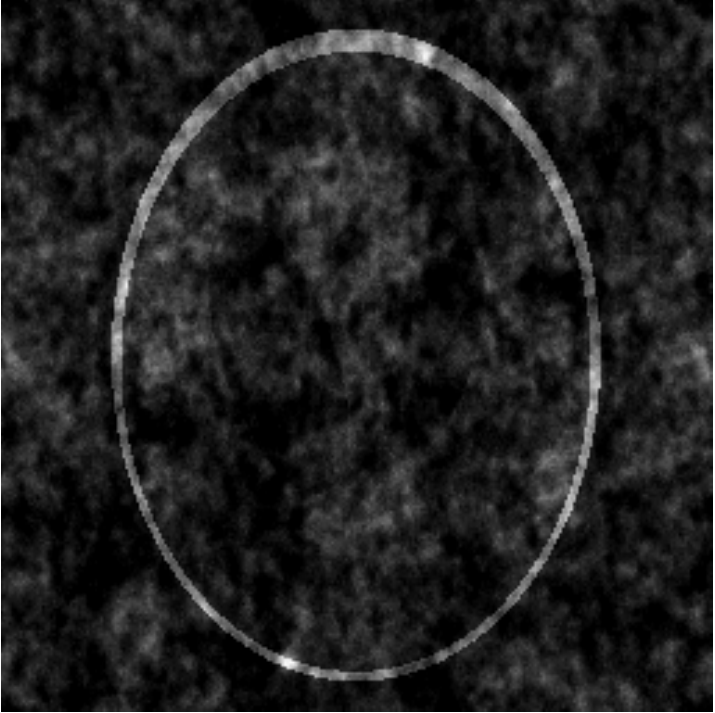}
\includegraphics[width=0.3\textwidth]{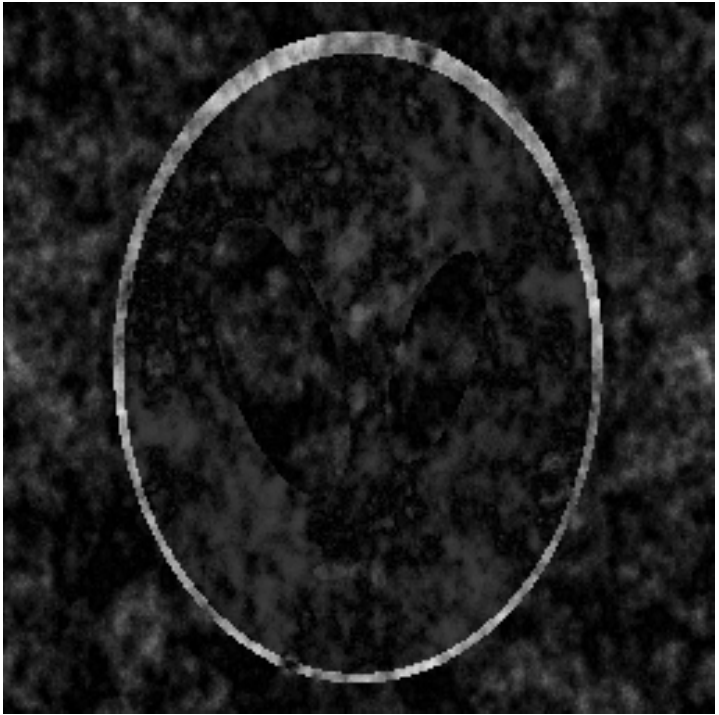}
\includegraphics[width=0.3\textwidth]{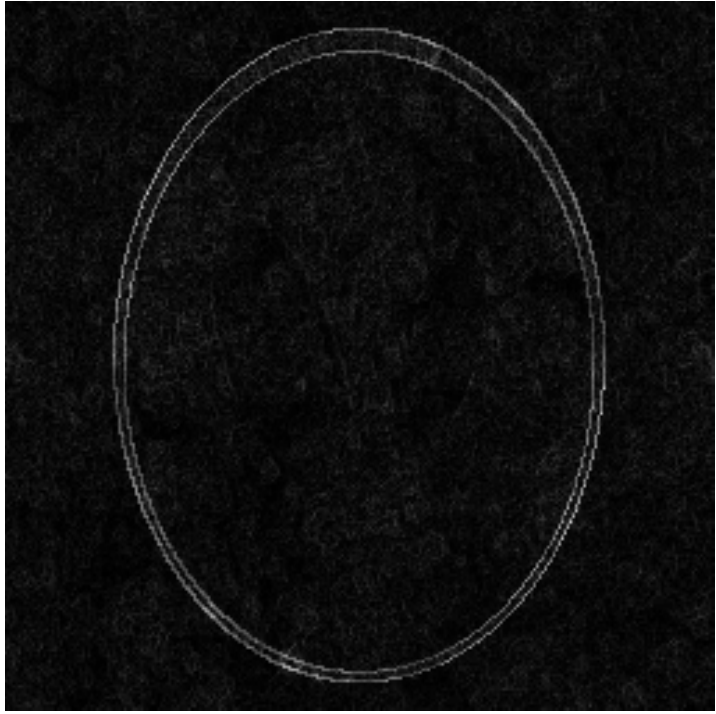}
\end{center}
\caption{Noiseless L1-min reconstructed image (left) and
the differences (middle)  from the original image. The plot on
the right is the gradient of the reconstructed image.}
\label{fig3}
\end{figure}

Because the image is not sparse, L1-min reconstruction
produces a poor result even in the absence of noise,
Figure \ref{fig3}. The relative error is $66.8\%$ in the $L2$ norm and $72.8\%$ in the TV norm.
  Only the outer boundary, which have the largest pixel values,  is reasonably recovered. 

Figure \ref{fig4} shows the results of TV-min reconstruction
in the presence of $5\%$ (top)  or $10\%$ (bottom)  noise.
Evidently, the performance is greatly improved. 
 \begin{figure}[h]
\begin{center}
\commentout{
\includegraphics[width=0.3\textwidth]{phantom-noise5.pdf}
\includegraphics[width=0.3\textwidth]{gradient-noise5.pdf}
\includegraphics[width=0.3\textwidth]{phantom-noise5-difference.pdf}\\
\includegraphics[width=0.3\textwidth]{phantom-noise10.pdf}
\includegraphics[width=0.3\textwidth]{gradient-noise10.pdf}
\includegraphics[width=0.3\textwidth]{phantom-noise10-difference.pdf}
}
\includegraphics[width=0.3\textwidth]{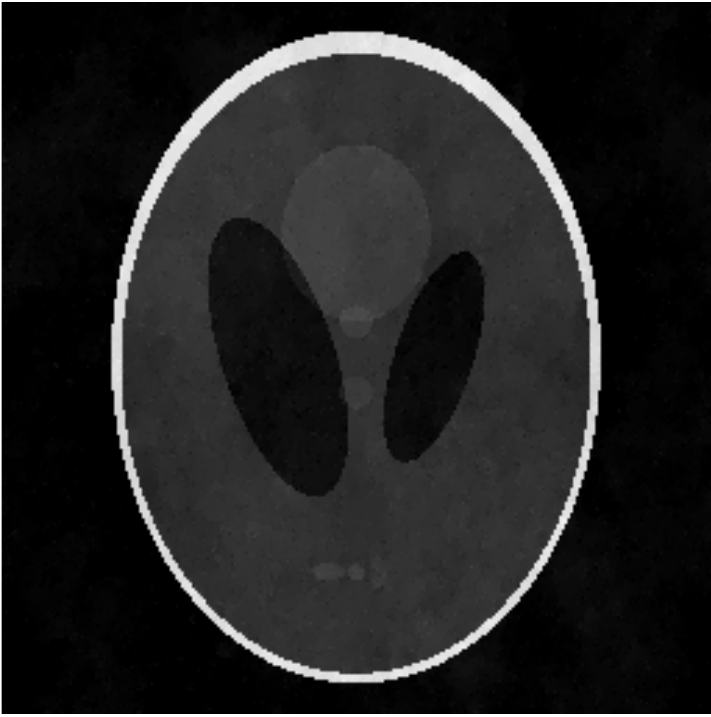}
\includegraphics[width=0.3\textwidth]{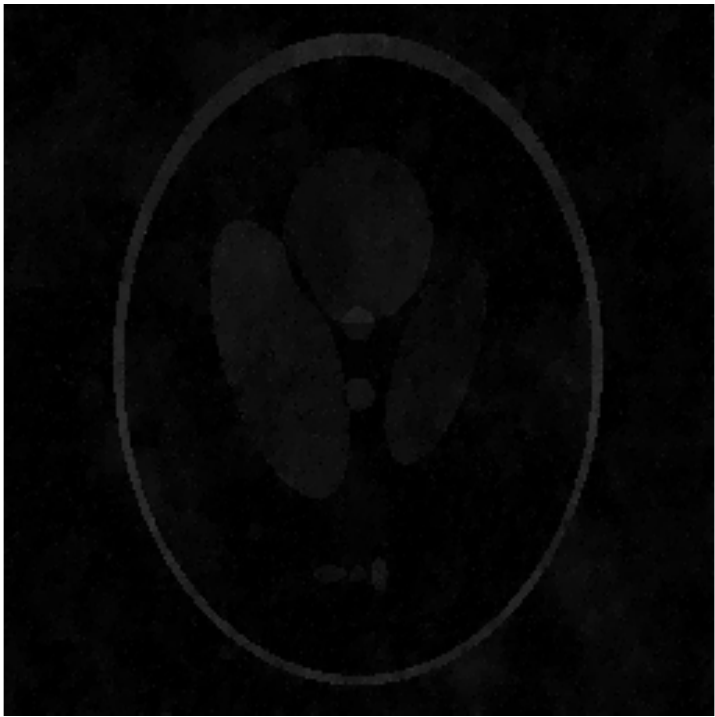}
\includegraphics[width=0.3\textwidth]{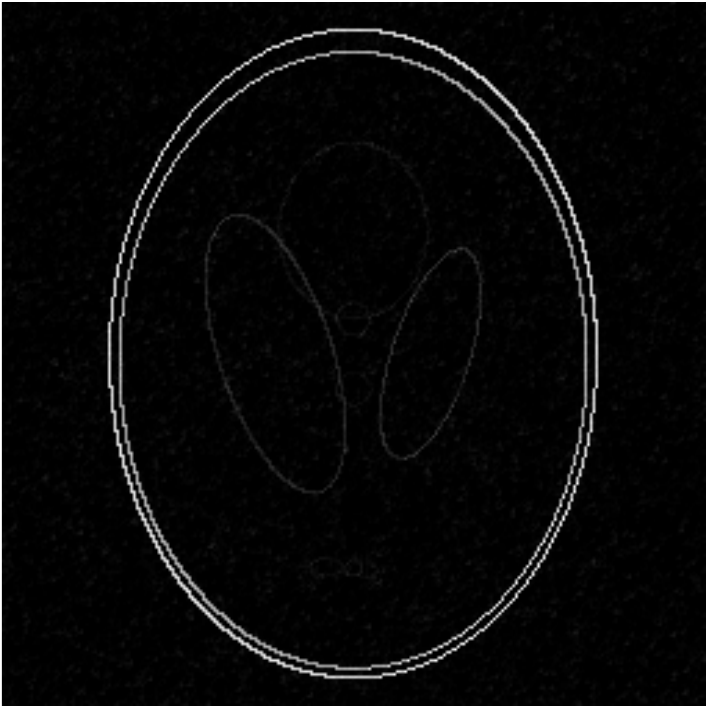}
\includegraphics[width=0.3\textwidth]{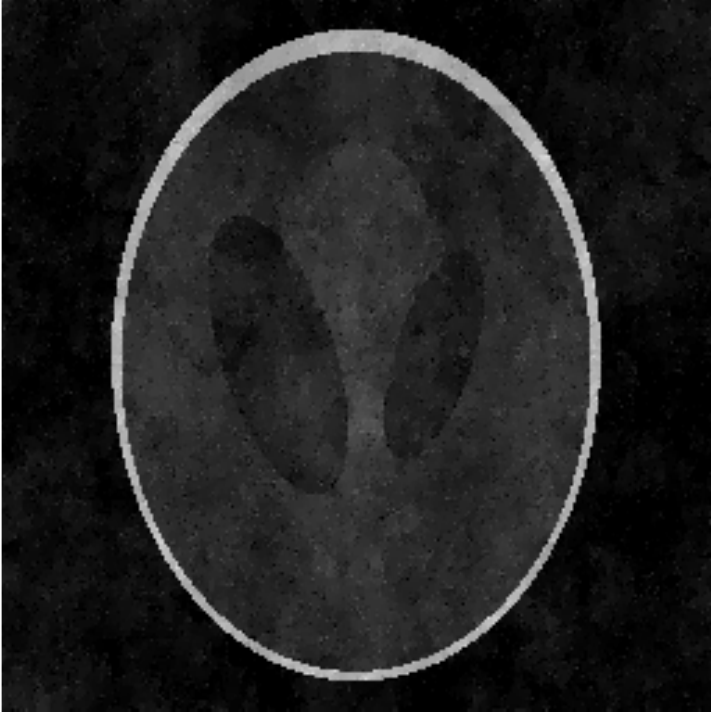}
\includegraphics[width=0.3\textwidth]{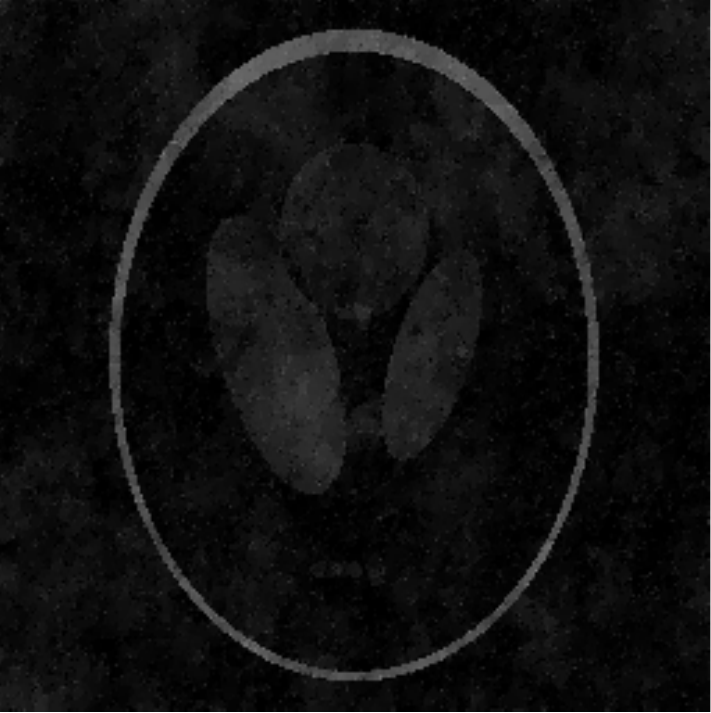}
\includegraphics[width=0.3\textwidth]{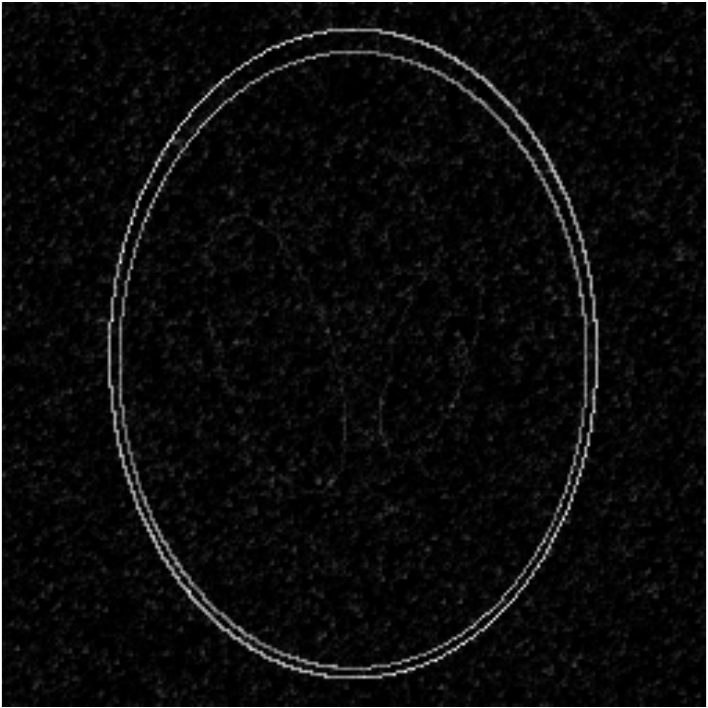}
\end{center}
\caption{TV-reconstructed image  with $5\%$ (top left)
and $10\%$ (bottom left)  and the respective differences (middle)  from the original image. The plots on the right column are the magnitudes of the reconstructed image gradients.  }
\label{fig4}
\end{figure}

\appendix
\section{Proof of Theorem \ref{thm:bp2} }\label{app:B}
The argument is patterned after  \cite{Can} with adaptation to the  CJS setting.
\begin{proposition} We have
\[
\lt|\Re\lan \varphi(\bZ),\varphi(\bZ')\ran\rt|
\leq \delta_{s+s'}\|\bZ\|_{2,2}\|\bZ'\|_{2,2}
\]
for all $\bZ, \bZ'$ supported on disjoint subsets $T,T'\subset \{1,...,m\} $ with $|S|\leq s, |S'|\leq s'.$
\label{lemma1}
\end{proposition}
\begin{proof} Without loss of generality, suppose that $\|\bZ\|_{2,2}=\|\bZ'\|_{2,2}=1$.  Since $\bZ\perp \bZ'$,  $\|\bZ\pm \bZ'\|^2_{2,2}=2.$ 
Hence we have from  the RIP  (\ref{rip})
\beq
2(1-\delta_{s+s'})\leq
\|\varphi(\bZ\pm \bZ')\|_{2,2}^2\leq 2(1+\delta_{s+s'})\label{a.1}
\eeq
By the parallelogram identity and (\ref{a.1}) 
\[
\lt|\Re\lan \varphi(\bZ), \varphi(\bZ')\ran\rt|=
{1\over 4}\lt|\|\varphi(\bZ)+\varphi(\bZ')\|_{2,2}^2-\|\varphi(\bZ)-\varphi(\bZ')\|_{2,2}^2\rt|
\leq \delta_{s+s'} \]
which proves the proposition.  

\end{proof}

By the triangle inequality and the fact that $\bX$ is in the  feasible
set we have
\beq
\label{a7}
\|\varphi(\hat \bX-\bX)\|_{2,2}\leq \|\varphi(\hat \bX)-\bY\|_{2,2}+\|\bY-\varphi(\bX)\|_{2,2}\leq
2\ep.
\eeq
Set $\hat \bX=\bX+\bD$ and decompose $\bD$ into a sum of 
$\bD_{S_0}, \bD_{S_1}, \bD_{S_2},...,$ each of row-sparsity at most $s$.
Here $S_0$ corresponds to the locations of the $s$ largest rows of $\bX$; $S_1$ the locations of the $s$ largest
rows of $\bD_{S_0^c}$;
$S_2$ the locations of the next $s$ largest rows of
$\bD_{S_0^c}$, and so on. 

{\bf Step (i)}. Define the norm
\[
\|\bZ\|_{\infty,2}=\max_{j}  \|\hbox{\rm row}_j (\bZ)\|_2. 
\]
For $j\geq 2$,
\beqn
\|\bD_{S_j}\|_{2,2}\leq s^{1/2} \|\bD_{S_j}\|_{\infty,2}\leq s^{-1/2}
\|\bD_{S_{j-1}}\|_{2,2}
\eeqn
and hence
\beq
\label{a6}
\sum_{j\geq 2} \|\bD_{S_j}\|_{2,2}\leq s^{-1/2} \sum_{j\geq 1}
\|\bD_{S_j}\|_{1,2}\leq s^{-1/2} \|\bD_{S_0^c}\|_{1,2}.
\eeq
This yields by the Cauchy-Schwarz 
inequality
\beq
\label{a.2}
\|\bD_{(S_0\cup S_1)^c}\|_{2,2}=\|\sum_{j\geq 2} \bD_{S_j}\|_{2,2}
\leq \sum_{j\geq 2} \|\bD_{S_j}\|_{2,2}\leq s^{-1/2} \|\bD_{S_0^c}\|_{1,2}.
\eeq
Also we have
\beqn
\|\bX\|_{1,2}&\geq &\|\hat \bX\|_{1,2}\\
&=&\|\bX_{S_0}+\bD_{S_0}\|_{1,2}
+\|\bX_{S_0^c}+\bD_{S_0^c}\|_{1,2}\\
&\geq&
\|\bX_{S_0}\|_{1,2}-\|\bD_{S_0}\|_{1,2}-\|\bX_{S_0^c}\|_{1,2}
+\|\bD_{S_0^c}\|_{1,2}
\eeqn
which  implies
\beq
\label{a.3}
\|\bD_{S_0^c}\|_{1,2}\leq 2\|\bX_{S_0^c}\|_{1,2}+\|\bD_{S_0}\|_{1,2}.
\eeq
Note that $\|\bX_{S_0^c}\|_{1,2}=\|\bX-\bX^{(s)}\|_{1,2}$ by definition.
Applying (\ref{a.2}), (\ref{a.3})  and the Cauchy-Schwartz inequality to $\|\bD_{S_0}\|_{1,2}$ gives
\beq
\label{a4}
\|\bD_{(S_0\cup S_1)^c}\|_{2,2}\leq \|\bD_{S_0}\|_{2,2}+2e_0
\eeq
where $e_0\equiv s^{-1/2} \|\bX-\bX^{(s)}\|_{1,2}$.

{\bf Step (ii).} Define the inner product 
\beqn
\lan {\mathbf A},{\mathbf B}\ran=
\sum_{i,j}  A^*_{ij}B_{ij}
\eeqn
Observe
\beq\label{38}
&&\|\varphi(\bD_{S_0\cup S_1})\|_{2,2}^2\\
&=&\lan \varphi(\bD_{S_0\cup S_1}),\varphi(\bD)\ran
-\lan \varphi(\bD_{S_0\cup S_1}),\sum_{j\geq 2}
\varphi(\bD_{S_j})\ran\nn\\
&=&\Re \lan\varphi(\bD_{S_0\cup S_1}), \varphi(\bD)\ran
-\sum_{j\geq 2}\Re \lan \varphi( \bD_{S_0\cup S_1}),
\varphi(\bD_{S_j})\ran\nn\\
&=&\Re \lan \varphi(\bD_{S_0\cup S_1}), \varphi(\bD)\ran
-\sum_{j\geq 2}\lt[\Re \lan\varphi(\bD_{S_0}),
\varphi(\bD_{S_j})\ran+ \Re \lan \varphi(\bD_{S_1}),
 \varphi(\bD_{S_j})\ran\rt].\nn
\eeq

From (\ref{a7}) and the RIP (\ref{rip})  it follows that
\beqn
|\lan \varphi(\bD_{S_0\cup S_1}), \varphi(\bD)\ran|&\leq \|\varphi( \bD_{S_0\cup S_1})\|_{2,2}
\| \varphi(\bD)\|_{2,2}
&\leq 2\ep \sqrt{1+\delta_{2s}}\|\bD_{S_0\cup S_1}\|_{2,2}.
\eeqn
Moreover, it follows from Proposition \ref{lemma1}
that
\beq
\label{102}\lt|\Re\lan \varphi( \bD_{S_0}), \varphi(\bD_{S_j})\ran \rt|
&\leq &\delta_{2s} \|\bD_{S_0}\|_{2,2} \|\bD_{S_j}\|_{2,2}\\
\label{103}\lt|\Re\lan \varphi(\bD_{S_1}), \varphi(\bD_{S_j})\ran \rt|
&\leq &\delta_{2s}  \|\bD_{S_0}\|_{2,2} \|\bD_{S_j}\|_{2,2}
\eeq
for $j\geq 2$. Since $S_0$ and $S_1$ are disjoint
\beqn
\|\bD_{S_0}\|_{2,2}+\|\bD_{S_1}\|_{2,2}\leq \sqrt{2}
\sqrt{\|\bD_{S_0}\|^2_{2,2}+\|\bD_{S_1}\|^2_{2,2}}
=\sqrt{2}  \|\bD_{S_0\cup S_1}\|_{2,2}.
\eeqn
Also by (\ref{38})-(\ref{103}) and RIP
\[
(1-\delta_{2s}) \|\bD_{S_0\cup S_1}\|_{2,2}^2
\leq \| \varphi(\bD_{S_0\cup S_1})\|_{2,2}^2\leq
\|\bD_{S_0\cup S_1}\|_{2,2}\lt(2\ep \sqrt{1+\delta_{2s}}
+\delta_{2s}  \sum_{j\geq 2}
\|\bD_{S_j}\|_{2,2}\rt). 
\]
Therefore from (\ref{a6}) we obtain
\beqn
\label{a8}
\|\bD_{S_0\cup S_1}\|_{2,2}\leq \alpha \ep +\rho s^{-1/2}
\|\bD_{S_0^c}\|_{1,2},\quad \alpha={2\sqrt{1+\delta_{2s}}\over
1-\delta_{2s}},\quad \rho={{\sqrt{2}}\delta_{2s}\over
1-\delta_{2s}}
\eeqn
and moreover by (\ref{a.3}) and the definition
of $e_0$ 
\[
\|\bD_{S_0\cup S_1}\|_{2,2}\leq \alpha \ep +\rho 
\|\bD_{S_0}\|_{2,2}+2\rho e_0 
\]
after applying the Cauchy-Schwartz inequality to bound $\|\bD_{S_0}\|_{1,2}$ by $s^{1/2}\|\bD_{S_0}\|_{2,2}$. 
Thus 
\[
\|\bD_{S_0\cup S_1}\|_{2,2}\leq (1-\rho)^{-1}(\alpha \ep+2\rho e_0)
\]
if (\ref{rip}) holds. 

Finally,
\beqn
\|\bD\|_{2,2}&\leq& \|\bD_{S_0\cup S_1}\|_{2,2}+\|\bD_{(S_0\cup S_1)^c}\|_{2,2}\\
&\leq& 2\|\bD_{S_0\cup S_1}\|_{2,2}+2 e_0\\
&\leq&
2(1-\rho)^{-1} (\alpha\ep+ (1+\rho)e_0)
\eeqn
which is the desired result.

\section{Proof of Theorem \ref{thm4}}
  
  We  prove the theorem  by induction. 
 
Let  $\hbox{supp} (\mbx)=\cS=\{J_1,\ldots, J_s\}$
and $$X_{\rm max}=\|\hbox{\rm row}_{J_1}(\bX)\|_1\geq \|\hbox{row}_{J_2} (\bX)\|_1\geq \cdots\geq\|\hbox{\rm row}_{J_s}(\bX)\|_1=X_{\rm min}$$. 
\commentout{
In the first step,
\beq
\nn  \sum_{j=1}^d|\Phi^*_{j,J_1}Y_j|^2 & =& \sum_{j=1}^d|{X}_{J_1j} + X_{J_2j}\Phi_{j,J_1}^*\Phi_{j,J_2} + ... + X_{J_sj}\Phi_{j,J_1}^{*}\Phi_{j,J_s} + \Phi^*_{j,J_1}E_j|^2 \\
 &\geq & \sum_{j=1}^d|{X}_{J_1j}|^2 -2 \sum_j|X^*_{J_1j}(X_{J_2j}\Phi_{j,J_1}^*\Phi_{j,J_2} + ... + X_{J_sj}\Phi_{j,J_1}^{*}\Phi_{j,J_s} + \Phi^*_{j,J_1}E_j)|\nn\\
   &\geq&  X^2_{\rm max}(1 - 2(s-1)\mu_{\rm max}) - 2X_{\rm max}\|\bE\|_{2,2}. \label{14}
\eeq 
On the other hand, $\forall l \notin\hbox{supp}(\mbx)$,
\beq
\label{15}   \sum_{j=1}^d |\Phi_{j,l}^*Y_j|^2  & =&\sum_{j=1}^d |X_{J_1j}\Phi^*_{j,l}\Phi_{j,J_1}+ X_{J_2j}\Phi_{j,l}^{*}\Phi_{j,J_2}+ ... + X_{J_sj}\Phi_{j,l}^{*}\Phi_{j,J_s} + \Phi_{j,l}^*E_j|^2 \\
  & \le& 2sX^2_{\rm max} \mu^2_{\rm max} +2\|\bE\|_{2,2}.\nn
\eeq
}
In the first step,
\beq
 \label{14}  \sum_{j=1}^d |\Phi^*_{j,J_1}Y_j| & =& \sum_{j=1}^d|{X}_{J_1j} + X_{J_2j}\Phi_{j,J_1}^*\Phi_{j,J_2} + ... + X_{J_sj}\Phi_{j,J_1}^{*}\Phi_{j,J_s} + \Phi^*_{j,J_1}E_j| \\
   &\geq&  X_{\rm max} - X_{\rm max}(s-1)\mu_{\rm max} - \sum_j\|E_j\|_2.\nn
\eeq 
On the other hand, $\forall l \notin\hbox{supp}(\mbx)$,
\beq
\label{15}   \sum_{j=1}^d |\Phi_{j,l}^*Y_j|  & =&\sum_{j=1}^d |X_{J_1j}\Phi^*_{j,l}\Phi_{j,J_1}+ X_{J_2j}\Phi_{j,l}^{*}\Phi_{j,J_2}+ ... + X_{J_sj}\Phi_{j,l}^{*}\Phi_{j,J_s} + \Phi_{j,l}^*E_j| \\
  & \le& X_{\rm max} s\mu_{\rm max} +\sum_j\|E_j\|_2.\nn
\eeq
Hence,  if
 \[
 (2s-1)\mu_{\rm max} +  \frac{2\sum_j\|E_j\|_2}{X_{\rm max}} < 1,
 \]
 then the right hand side of (\ref{14}) is greater than
 the right hand side of (\ref{15}) which implies that
 the first  index selected by OMP must belong to $\hbox{supp}(\mbx)$. 

To continue the induction process, we  state the straightforward generalization of a standard  uniqueness result for sparse recovery to
the joint sparsity setting (Lemma 5.3,  \cite{DET}). 

\begin{proposition}
Let $\bZ=\varphi(\bX)$ and $\bY=\bZ+\bE$. Let $\cS^k$ be a set of $k$ indices and let $\mathbf{A}\in\IC^{n\times d}$ with $\supp{(\mathbf{A})}=\cS^k$. Define
\beq
\label{1001}
\bY'=\bY-\varphi(\mathbf{A})
\eeq
and
\[
\bZ'=\bZ-\varphi(\mathbf{A}).
\]
Clearly, 
$\bY'=\bZ'+\bE$. 
If $\cS^k\subsetneq \hbox{supp} (\mbx)$ and  the sparsity $s$ of $\bX$ satisfies $2s<1+\mu_{\rm max}^{-1}$, then
$\bZ'$ has a unique sparsest representation
$\bZ'=\varphi(\bX')$ with the sparsity of $\bX'$ at most $s$.
\label{prop22}
\end{proposition}

Proposition \ref{prop22} says  that selection of a column, followed by the formation of the residual signal, leads to a situation like before, where the ideal noiseless signal has no more representing columns than before, and the noise level is the same.

Suppose
that the set $\cS^k\subseteq \supp (\mbx)$ of $k$ distinct indices
has been selected and
that $\mathbf{A}$ in Proposition \ref{prop22} solves the following least squares problem
\beq
\label{ls}
\mathbf{A}=\hbox{arg}\min\|\bY-\bA \bB\|_{2,2},\quad \hbox{s.t.}\quad  \supp (\bB)\subseteq \cS^k
\eeq
{\em without}  imposing the constraint $\cL$.
This is equivalent to the concatenation $\mathbf{A}=[A_j]$ of 
$d$ separate least squares solutions
\beq
\label{ls'}
{A}_j=\hbox{arg}\min_{B_j}\|Y_j-\bA_j B_j\|_{2},\quad \hbox{s.t.}\quad  \supp (B_j)\subseteq \cS^k
\eeq

Let $\bPhi_{j,\cS^k}$ be the column submatrix of $\bA_j$ indexed by the set $\cS^k$. By (\ref{1001}) and (\ref{ls'}),  $\bPhi^*_{j,\cS^k}Y'_j=0,\forall j,$ which implies
that no element of $\cS^k$ gets  selected at
the $(k+1)$-st step. 

In order to ensure that some element in $\supp (\mbx) \setminus \cS^k$ gets selected at the $(k+1)$-st
step we only need to repeat the calculation (\ref{14})-(\ref{15}) to obtain the condition 
\beq
\label{1002}
(2s-1)\mu_{\rm max} + \frac{2\sum_j\|E_j\|_2}{\|X_{J_{k+1}}\|_1} < 1.
\eeq
Since $\sum_j\|E_j\|_2\leq \sqrt{d} \|\bE\|_{2,2}=\sqrt{d}\ep$
(\ref{1002})  follows from  
\beq
\label{snr}
(2s-1)\mu_{\rm max} + \frac{2\sqrt{d}\ep}{X_{\rm min}} < 1
\eeq
which is the same as (\ref{sparse}) and allows us to apply
Proposition \ref{prop22} repeatedly. 

By the $s$-th step, all elements of the support set
are selected and by the nature of the least squares
solution the $2$-norm of the residual is at most $\ep$. 
Thus the stopping criterion is met and the iteration
stops after $s$ steps.

On the other hand, it follows from the calculation
\beqn
\sum_j\|Y'_j\|_2 &\geq & \sum^d_{j=1}\big| \Phi^*_{j,J_{k+1}} Y'_j\big|\\
&=&\sum_{j}\big| X_{J_{k+1}j}+\sum_{i=k+2}^s X_{J_ii} \Phi_{j,J_{k+1}}^*
\Phi_{i,J_i}+\Phi_{j,J_{k+1}}^*E_j\big|\\
&\geq& \|\hbox{\rm row}_{J_{k+1}}(\bX)\|_1-\mu_{\rm max} (s-k-1) \|\hbox{\rm row}_{J_{k+2}}(\bX)\|_1
-\sum_j\|E_j\|_2\\
&\geq& (1-\mu_{\rm max} (s-k-1))\|\hbox{\rm row}_{J_{k+1}}(\bX)\|_1-\sum_j\|E_j\|_2
\eeqn
and (\ref{snr}) (equivalently, 
$X_{\rm min}(1-\mu_{\rm max}(2s-1))>2\sqrt{d}\ep$)
that $\|\bY\|_{1,2}> \sqrt{d}\ep$ for $k=0,1,\cdots, s-1$. Thus
the iteration does not stop until $k=s$. 

Since $\hat\bX$ be the solution of the least squares problem
(\ref{ls2}), we have 
\[
\|\bY-\bA\hat \bX\|_{2, 2}\leq \|\bY-\bA \bX\|_{2, 2}  \leq\ep
\]
 and
\[
\|\bA(\bX-\hat \bX)\|^2_{2,2}\leq 2 \|\bY-\bA\bX\|^2_{2,2}+2\|\bY-\bA\hat \bX\|^2_{2,2}\leq 2 \ep^2
\]
which implies
\[
\|\hat \bX-\bX\|_{2,2}\leq \sqrt{2}\ep/\lambda_{\rm min}
\]
where 
\[
\lambda_{\rm min}= \min_j \{ \hbox{\rm the $s$-th singular
value of the column submatrix of $\bPhi_{j}$ indexed by
$\cS$}\}
\].

The desired error bound (\ref{err}) can now be obtained from
 the following result (Lemma 2.2, \cite{DET}). 
\begin{proposition} 
 Suppose $s<1+\mu(\bA_j)^{-1}$. Every $m\times s$ column submatrix of $\bPhi_j$ has the $s$-th singular value bounded below by $ \sqrt{1-\mu(\bA_j) (s-1)}$.
 \label{prop4}
\end{proposition}

By Proposition \ref{prop4},  $\lambda_{\rm min}
\geq  \sqrt{1-\mu_{\rm max} (s-1)}$ and thus 
\[
\|\hat \bX-\bX\|_{2,2}\leq {\sqrt{2}\ep\over \sqrt{1-\mu_{\rm max} (s-1)}}.
\] \\

\bigskip
{\bf Acknowledgement.} I thank  Stan Osher and
Justin Romberg for suggestion of  publishing  this note
at the IPAM workshop ``Challenges in Synthetic Aperture Radar" February 6-10, 2012. 
I thank the anonymous referees and Deanna Needell for pointing out the reference \cite{NW} which
helps me appreciate more deeply the strength and weakness of my approach.
I am grateful to  Wenjing Liao
for preparing Fig. 2-4. The research is partially supported by
the U.S. National Science Foundation  under grant DMS - 0908535.

\end{document}